\documentclass[11pt]{amsart}
\usepackage{amssymb,amscd,amsmath}
\usepackage[all]{xy}
\newtheorem{thm}{Theorem}

\newtheorem{cor}[thm]{Corollary}

\newtheorem{prop}[thm]{Proposition}

\newtheorem{lem}[thm]{Lemma}

\theoremstyle{remark}

\theoremstyle{definition}

\newcommand{\F}{{\mathbb{F}}}

\newcommand{\Z}{{\mathbb{Z}}}

\newcommand{\GL}{\mathrm{GL}}
\newcommand{\SL}{\mathrm{SL}}

\newcommand{\Alt}{{\raise 2pt\hbox{$\scriptstyle\bigwedge$}}}
\newcommand{\AAA}{\mathsf{A}}
\newcommand{\SSS}{\mathsf{S}}

\newcommand{\la}{\lambda}

\begin{document}
\title[Partition length character bounds]
{Sharp character bounds for symmetric groups in terms of partition length}

\author{Michael Larsen}
\email{mjlarsen@iu.edu}
\address{Department of Mathematics\\
    Indiana University \\
    Bloomington, IN 47405\\
    U.S.A.}

\thanks{Michael Larsen was partially supported by NSF grant DMS-2401098 and Simons grant MPS-TSM-00007695.}

\begin{abstract}
Let $\SSS_n$ denote a symmetric group, $\chi$ an irreducible character of $\SSS_n$, and $g\in \SSS_n$ an element which decomposes into $k$ disjoint cycles, where $1$-cycles are included.
Then $|\chi(g)|\le k!$, and this upper bound is sharp for fixed $k$ and varying $n$, $\chi$, and $g$.  This implies a sharp upper bound of $k!$ for unipotent character values of $\SL_n(q)$ at regular semisimple elements with characteristic polynomial $P(t)=P_1(t)\cdots P_k(t)$, where the $P_i$ are irreducible over $\F_q[t]$.

\end{abstract}

\maketitle

Let $k$ be a fixed positive integer.  Let $n\ge k$ be an integer and  $g$ an element of the symmetric group $\SSS_n$ whose orbits partition $\{1,2,\ldots,n\}$ into $k$ disjoint sets.
Let $\chi$ denote an irreducible character of $\SSS_n$.  A theorem of Larsen and Shalev \cite[Theorem~7.2]{LS} asserts that 
\begin{equation}
\label{LS}
|\chi(g)| \le 2^{k-1}k!.  
\end{equation}
The fact that the right hand side of \eqref{LS} does not depend on $n$ is rather special to $\SSS_n$.
The centralizer bound for $g$ gives only $O(n^{k/2})$, and that is the most that is actually true if $g$ is assumed to be an even permutation and $\chi$ is taken to be
an irreducible character of $\AAA_n$ rather than $\SSS_n$.  (See \cite[Theorem 2.5.13]{JK}.)
However, \eqref{LS} is not the best possible bound.  The main result of this note is the following improvement:

\begin{thm}
\label{Main}
For $k$, $n$, $g$, and $\chi$ as above,
$$|\chi(g)| \le k!$$
\end{thm}

This estimate is optimal:

\begin{thm}
\label{Sharp}
For all $k\ge 1$, there exist $n$, $g\in \SSS_n$, and $\chi$ an irreducible character of $\SSS_n$ such that
$g$ partitions $\{1,\ldots,n\}$ into $k$ orbits, and $|\chi(g)| = k!$.
\end{thm}

When $g$ has at least one fixed point, however, there is a stronger estimate.

\begin{thm}
\label{Fixed}
For $k$, $n$, $g$, and $\chi$ as above, if $g$ has at least one fixed point, then
$$|\chi(g)| \le (k-1)!,$$
and this bound is sharp for elements $g$ with $k$ orbits, of which one at least has size $1$.
\end{thm}

As an application of Theorems \ref{Main} and \ref{Sharp}, we obtain the optimal upper bound for unipotent character values of $\SL_n(\F_q)$ at regular semisimple elements.

\begin{cor}
\label{SL}
Let $k$ be a positive integer.  For all $n\ge k$, all prime powers $q$, all unipotent characters $\chi$ of $\SL_n(q)$,
and all regular semisimple elements $g\in \SL_n(q)$ whose characteristic polynomial decomposes into $k$ irreducible factors, we have
$$|\chi(g)| \le k!.$$
Moreover, this upper bound is sharp.
\end{cor}

The irreducible characters $\chi$ of $\SSS_n$ are labeled by partitions $\la\vdash n$, where $\lambda_1 \ge \lambda_2\ge \cdots \ge\lambda_r > \lambda_{r+1}=0$.  When we wish to specify the partition of an irreducible character, we write $\chi=\chi_\la$.
To each $\la$, we associate the Young diagram
$$Y_\la = \{(i,j)\in \Z^2\mid i\in [1,r],\ j\in [1, \lambda_i]\}.$$
We regard $Y_\la$ as a graph in which two vertices are connected if their Euclidean distance is $1$.
The \emph{rim} $R_\la$ of $Y_\la$ consists of $(x,y)\in Y_\la$ such that $(x+1,y+1)\not\in Y_\la$.
A subset $H\subset R_\la$ is said to be a \emph{rim hook} if it satisfies the following two properties:
\begin{itemize}
\item $H$ is a connected subset of $R_\la$.
\item The complement of $H$ in $Y_\la$ is the Young diagram $Y_\mu$ of some partition $\mu$.
\end{itemize}
Note that we can express the second condition as saying that, assuming $(x,y)\in H$, if $(x+1,y)$ belongs to $Y_\la$, then it belongs to $H$, and likewise for $(x,y+1)$.

The elements $(x,y)$ of $R_\la$ can be arranged by increasing lexicographic order on $(x,-y)$:
\begin{equation}
\label{rim}
(1,\lambda_1),(1,\lambda_1-1),\ldots,(1,\lambda_2),(2,\lambda_2),\ldots,(2,\lambda_3),(3,\lambda_3),\ldots
\end{equation}
Note that the value of $x-y$ increases by $1$ at each step and that two elements of $R_\la$ are connected if and only if they are consecutive terms in \eqref{rim}.  All rim hooks are therefore consecutive subsequences of \eqref{rim}.
The condition that the complement of $H$ in $Y_\la$ is a Young diagram is that the first term in the rim hook is of the form $(i,\lambda_i)$ for some $i$ and the last term is of the form $(\lambda'_j,j)$ for some $j$, where $\lambda'$ is the dual partition to $\lambda$.


We define the $i$th \emph{layer} of $Y_\la$ as the set of $(x,y)\in Y_\la$ with $\min(x,y) = i$.
If $k$ is the largest integer such that $(k,k)\in Y_\la$, then $Y_\la$ has $k$ layers.
The elements of $Y_\la$ of the form $(i,\lambda_i)$, where $1\le i\le k$, or of the form $(\lambda'_j,j)$, where $1\le j\le k$, are \emph{special}, so there are $2k$ special elements in all.  If the first term of a rim hook occurs in \eqref{rim} before $(k,k)$,
then it is special, and likewise for the last term of a rim hook if it appears after $(k,k)$.  We say a rim hook is \emph{central} if it contains $(k,k)$, in which case both of its endpoints are special.  A rim hook which is not central is either \emph{left} or \emph{right}.

If we fix a positive integer $r$ and consider rim hooks of length $r$, each special point of the form $(i,\lambda_i)$ for $1\le i\le k$ is the left endpoint of at most one such rim hook, and each special point of the form $(\lambda'_j,j)$ for $1\le j\le k$ is the right endpoint of at most one such rim hook.

\begin{lem}
\label{hooks}
Let $H$ be a rim hook of length $r$.
\begin{enumerate}
\item If $H$ is non-central, $1\le i\le k$, and $(i,\lambda_i)$ is the left endpoint of $H$, then
$r\le \lambda_i - i$.
\item  If $H$ is non-central, $1\le j\le k$, and $(\lambda'_j,j)$ is the right endpoint of $H$, then
then $r\le \lambda'_j-j$.
\item If $H$ is central, $(i,\lambda_i)$ is the left endpoint and $(\lambda'_j,j)$ the right endpoint of $H$, then
$r = 1+\lambda_i-i+\lambda'_j-j$.
\end{enumerate}
\end{lem}

\begin{proof}
Since the value of $x-y$ increases by one at each step of \eqref{rim}, it follows
that the consecutive sequence from $(i,\lambda_i)$ to $(k,k)$ has $1+\lambda_i-i$ terms.  The consecutive subsequence of length $r$ starting at $(i,\lambda_i)$ does not include $(k,k)$, so $r\le \lambda_i-i$, which gives part (1).  Part (2) is exactly the same, and for part (3), we observe that the length of the consecutive subsequence of \eqref{rim} from $(i,\lambda_i)$ to $(\lambda'_j,j)$ can be computed by subtracting $i-\lambda_i$ from $\lambda'_j-j$ and adding $1$.
\end{proof}

\begin{lem}
\label{key ineq}
If there are $t$ rim hooks of length $r$, then $tr \le n$.
\end{lem}

\begin{proof}
Let $C_r$ denote the number of central rim hooks of length $r$.  Since each such rim hook is determined by its left enpoint, which must be of the form $(i,\lambda_i)$ for $1\le i\le k$,  so $C_r\le k$.

Taking the left endpoint of each left rim hook of length $r$, the right endpoint of each right rim hook of length $r$, and both endpoints of each central rim hook, all the terms in question are special, and no two are the same.  Summing the expression $\lambda_i-i$ for each $(i,\lambda_i)$ which appears as a left endpoint in this way, together with the expression $\lambda'_j-j$ for each $(\lambda'_j,j)$
that appears as a right endpoint, Lemma~\ref{hooks} implies the result is at least $tr$.  On the other hand, since $\lambda_i-i \ge 0$
and $\lambda'_j \ge j$ for all $i,j\le k$, the sum is at most
$$\sum_{i=1}^k (\lambda_i-i) + \sum_{j=1}^k (\lambda'_j-j) + C_r = n-k + C_r\le n.$$
\end{proof}

We next recall the Murnaghan-Nakayama Rule \cite[2.4.7]{JK}.

\begin{thm}\label{murn} {\rm (Murnaghan-Nakayama Rule)} 
Let $\rho \pi \in \SSS_n$, where
$\rho$ is an $r$-cycle and $\pi$ is a permutation of the $n-r$ fixed
points of $\rho$. Then 
\[
\chi_\la (\rho \pi) = \sum_\nu (-1)^{l(\nu)} \chi_{\la \backslash \nu} (\pi),
\]
where the sum is over all rim hooks $\nu$ of length $r$ in a $\la$-diagram,
and $l(\nu)$ is one less than the number of rows in the rim hook $\nu$.
\end{thm}

We can now prove Theorem~\ref{Main}.

\begin{proof}
We use induction on $k$.  The case $k=1$ is well known and is an immediate consequence of Theorem~\ref{murn}.
Suppose the theorem is known up to $k-1$.  Let $g$ be a permutation of $\{1,2,\ldots,n\}$ with $k$ orbits.  At least one orbit must have size $r\ge n/k$, so we can write $g$ as the product of an $r$-cycle and a permutation of the remaining $n-r$
elements, with a total of $k-1$ orbits.  By Theorem~\ref{murn},
$$|\chi(g)| = |\chi_\lambda(g)| \le \sum_\nu |\chi_{\lambda\setminus\nu}(g)|.$$

By the induction hypothesis, each summand is bounded above by $(k-1)!$.  By Lemma~\ref{key ineq}, $tr \le n$, so $t\le k$, and we are done.
\end{proof}

To present examples showing that Theorem~\ref{Main} is sharp, 
we adopt the notation $(a_1,\ldots,a_k | b_1,\ldots b_k)$ for
the partition with $k$ layers for which $\lambda_i = a_i+i$ and $\lambda'_j = b_j + j$ for all $i,j\le k$.

\begin{lem}
\label{symmetric}
Let $k$ be a positive integer, $a_1>a_2>\cdots >a_k\ge 0$ a decreasing sequence of non-negative integers
and $r \ge a_1+1$ a positive integer.  Define 
$$b_j = r-1-a_{k+1-j}$$
and
\begin{equation}
\label{a to lambda}
\la = (a_1,\ldots,a_k|b_1,\ldots,b_k) = (a_1,\ldots,a_k|r+1-a_k,\ldots,r+1-a_1).
\end{equation}
There are $k$ rim hooks of length $r$ in $Y_\la$, and removing any one of them from $Y_\la$ leaves a Young diagram of the form $Y_\mu$ where for some $j\in [1,k]$, 
\begin{equation}
\label{omit}
\mu= (a_1,\ldots,\widehat{a_p},\ldots,a_k| r-1-a_k,\ldots,\widehat{r-1-a_p},\ldots, r-1-a_1).
\end{equation}
\end{lem}

\begin{proof}
By Lemma~\ref{hooks} (3), the rim hook from $(p,\lambda_p)$ to $(\lambda'_{k+1-p},k+1-p)$ has length
$$1+\lambda_p-p+\lambda'_{k+1-p}-(k+1-p) = 1 + a_p + b_{k+1-p} = 1 + a_p + (r-1-a_p) = r.$$
Removing it from $Y_\la$ leaves $Y_\mu$, with 
$$\mu_i = \begin{cases}\lambda_i&\text{ if }i<p\\ \lambda_{i+1}-1&\text{ if }i\ge p.\\ \end{cases}$$
Writing
$$\mu = (c_1,\ldots,c_{k-1}|d_1,\ldots,d_{k-1}),$$
therefore, we have
$$c_i = \begin{cases}a_i&\text{ if }i<p\\ a_{i+1}&\text{ if }i\ge p.\\ \end{cases}$$
Likewise,
$$d_j =  \begin{cases}b_j&\text{ if }j<k-p\\ b_{j+1}&\text{ if }j\ge k-p.\\ \end{cases}$$
It follows that $\mu$ is given by \eqref{omit}.
\end{proof}

Theorem~\ref{Sharp} follows immediately from the following proposition:

\begin{prop}
\label{decreasing odd}
Let $k$ be a positive integer, $a_1>a_2>\cdots >a_k\ge 0$ a decreasing sequence of odd integers,
and $r > a_1$ an odd integer.  If $\la$ is given by \eqref{a to lambda}, $\chi$ is the irreducible
representation of $\SSS_{k r}$ associated to $\la$, and $g\in \SSS_{k r}$ consists of $k$ $r$-cycles, then
$$\chi(g) = (-1)^{\binom k2} k!.$$
\end{prop}

\begin{proof}
We use induction on $k$, where the base case $k=0$ is trivial.
The leg length of every rim hook of size $r$ is of the form 
$$\lambda_i - (k+1-i) \equiv a_i-k-1 \equiv k \pmod2.$$
Let $r^k$ denote any element in $\SSS_{kr}$ consisting of $k$ $r$-cycles.
By Theorem~\ref{murn},
\begin{equation}
\label{r to the k}
\chi_\lambda(r^k) = (-1)^k \sum_\nu \chi_{\lambda\setminus\nu}(r^{k-1}),
\end{equation}
where the sum is taken over all $k$ rim hooks of $\la$ of length $r$.
By Lemma~\ref{symmetric}, each $\chi_{\lambda\setminus\nu}$ satisfies the hypotheses of the proposition for $k-1$,
so the induction hypothesis and \eqref{r to the k} imply
$$\chi_\lambda(r^k) = (-1)^k k (-1)^{\binom{k-1}2} (k-1)! = (-1)^{\binom k2} k!.$$
The proposition follows.
\end{proof}
We remark that the simplest case in which the hypotheses of Proposition~\ref{decreasing odd} are satisfied is
$(2k-1,2k-3,\ldots,1|2k-1,2k-3,\ldots,1)$, which gives $\lambda = (2k,2k-1,\ldots,2,1)$.

The proof of Theorem~\ref{Fixed} follows the proofs of Theorem \ref{Main} and \ref{Sharp} closely.

\begin{proof}
We prove the theorem by induction on $k$.
It is trivial for $k=1$, so we assume $k\ge 2$, which implies $k < n$ and therefore $\frac k{k-1} > \frac n{n-1}$.  
As $g$ has $k-1$ orbits whose union includes all but one element, we may choose an orbit with $r\ge \frac{n-1}{k-1}$ elements.
If $t$ is the number of rim hooks of length $r$,
by Lemma~\ref{key ineq}, we have $tr \le n$, so
$$t\le \frac nr \le \frac{n(k-1)}{n-1} < k.$$
As $t$ is an integer, it is bounded above by $k-1$.
Using Murnaghan-Nakayama, we see that $|\chi(g)|$ is bounded above by a sum of $t\le k-1$ terms,
each of which, by the induction hypothesis, has absolute value $\le (k-2)!$, so the first claim follows by induction.

For the second claim, we take a decreasing odd sequence $a_1>a_2>\cdots > a_{k-1}$, we take $r\ge a_1+1$, and we define 
$$\lambda = (a_1,a_2,\ldots,a_{k-1},0|r+1-a_k,\ldots,r+1-a_1,0),$$
and take an element $g$ of type $1^1 r^{k-1}$ in $\SSS_{1+r(k-1)}$.
We apply Theorem~\ref{murn} with $\rho$ a $1$-cycle.  Removing any rim hook of length $1$ other than the singleton $(k,k)$ leaves a Young diagram with $k$ layers, so the
value of the corresponding character at any element with $k$ orbits is, again by Theorem~\ref{murn}, $0$.  Therefore, the only contribution to the sum comes from the value at
$r^{k-1}$ of the character associated to $(a_1,a_2,\ldots,a_{k-1}|r+1-a_k,\ldots,r+1-a_1)$, which has absolute value $(k-1)!$ by Proposition~\ref{decreasing odd}.

\end{proof}

We conclude with Corollary~\ref{SL}.

\begin{proof}
The action of the $q$-Frobenius on the roots of the characteristic polynomial $P(t)$ of $g$ gives an element $\sigma$ of $\SSS_n$
which has a total of $k$ cycles, where $k$ is the number of irreducible factors of $P(t)$.
By \cite[Proposition~3.3]{LM}, the value on $n$ of the unipotent character $\chi$ of $\SL_n(q)$ corresponding to a partition $\lambda$ is the same as the value on $\sigma$ of the irreducible character of $\SL_n(q)$ corresponding to $\lambda$,
so the corollary follows immediately from Theorems \ref{Main} and \ref{Sharp}.  

\end{proof}

We remark that even if $k=1$, an upper bound on $|\chi(g)|$ which holds for all irreducible characters of $\SL_n(q)$
must depend on $n$.  Indeed, if the characteristic polynomial of $g$ has a root which generates the multiplicative group
of $\F_{q^n}^\times$, then by \cite[p.~431]{Green}, for each primitive $q^n-1$ root of unity $\zeta$, there exists an
irreducible character $\chi$ of $\GL_n(q)$ such that 
$$\chi(g) = \zeta+\zeta^q+\zeta^{q^2}+\cdots+\zeta^{q^{n-1}}.$$
Taking $\zeta = e^{\frac{2\pi i}{q^n-1}}$, in the large $n$ limit this approaches
$$n - \sum_{j=1}^\infty (1-e^{2\pi i/q^j}) = n + O(1).$$
On the other hand, the centralizer $\F_{q^n}^\times\subset \GL_n(q)$ of $g$ in $\GL_n(q)$ has elements of all determinants, so the $\SL_n(q)$-conjugacy class of $g$ has the same cardinality as its $\GL_n(q)$-conjugacy class.
Therefore, every irreducible character of $\GL_n(q)$ which takes a non-zero value on $g$ remains irreducible on restriction to $\SL_n(q)$, which means that we cannot hope for a bound for general irreducible characters on $g$ significantly better than $|\chi(g)| \le n$.  In fact, for each $k$, we do have such a bound \cite[Corollary 7.6]{LTT}: 
$$|\chi(g)| \le k! \cdot n^k.$$

\end{document}